\RequirePackage{fix-cm}
\documentclass[12pt,english]{article}
\usepackage{amsmath}
\usepackage[T1]{fontenc}
\usepackage[letterpaper]{geometry}
\usepackage{amssymb}
\usepackage{comment}
\usepackage{latexsym}

\usepackage{graphicx}

\usepackage[backref,colorlinks=true]{hyper ref}
\usepackage{enumerate}
\usepackage[dvipsnames,usenames]{color}
\newcommand{\mc}[1]{\ensuremath{\mathcal{#1}}} 
\newcommand{\EE}{{\mathbb E}}
\newcommand{\po}{\mathcal{P}}
\definecolor{highlight}{rgb}{1,0.2,0.2}

\newtheorem{open}{Open Question}

\def\BBox{\kern  -0.2cm\hbox{\vrule width 0.2cm height 0.2cm}}

\newtheorem{lemma}{Lemma}[section]
\newtheorem{theorem}{Theorem}[section]

\newtheorem{acknowledgements}[theorem]{Acknowledgements}

\begin{document}

\title{Polytopes with Preassigned Automorphism Groups}
\author{
Egon Schulte\thanks{Email: schulte@neu.edu}\\
Northeastern University,\\
Department of Mathematics,\\
Boston, MA 02115, USA\\[.15in]
and\\[.15in] 
Gordon Ian Williams\thanks{Email:\ giwilliams@alaska.edu}\\
University of Alaska Fairbanks\\
Department of Mathematics and Statistics,\\
Fairbanks, AK 99709, USA}

\maketitle

\begin{abstract}
\noindent
{We prove that every finite group is the automorphism group of a finite abstract polytope isomorphic to a face-to-face tessellation of a sphere by topological copies of convex polytopes. We also} show that this abstract polytope may be realized as a convex polytope.\\[.2in]
\end{abstract}

\section{Introduction}\label{intro}
The study of the symmetry properties of convex polyhedra goes back to antiquity with the classification of the Platonic solids as the solids with regular polygonal faces possessing the greatest amount of possible symmetry. Subsequent investigation of {\em regular} polyhedra --- those in which any incident triple of vertex, edge and facet could be mapped to another by an isometry of the figure --- led to a number of interesting classes and relaxations of the definition of a polyhedron to encompass other types of regular polyhedra such as the Kepler-Poinsot polyhedra (allowing various kinds of self intersection), the Coxeter-Petrie polyhedra (allowing polyhedra with infinitely many faces)\cite{Cox73} and the Gr\"unbaum-Dress polyhedra (allowing nonplanar and non-finite faces)\cite{Gru77,Dre81,Dre85}. 

A particularly fruitful area of investigation during the last thirty years has been the study of abstract polytopes.  These are partially ordered sets which inherit their structure from derived constraints on convex polytopes. Of particular interest are the regular abstract polytopes --- those in which any two maximal chains may be mapped to each other by an automorphism of the poset --- whose structure is completely determined by their automorphism groups \cite{McMSch02}, and their close cousins the chiral polytopes~\cite{Pel12,SW1}. Both of these classes of abstract polytope have very restricted classes of groups from which their automorphism groups may be selected. This suggests a very natural question: which finite groups arise as the automorphism groups of polytopes, {
 specifically spherical or convex polytopes}? In the context of the current work we answer {\em all of them!}

\section{Basic Notions}
 For the convenience of the reader and to introduce  necessary notation, we recall some basic definitions from the theory of convex and abstract polytopes (standard references are \cite{Zie95,Gru03,McMSch02}). Let $\EE^{n}$ denote Euclidean $n$-space. A {\em convex polytope} $P$ is the {\em convex hull} of a finite collection of points in $\EE^{n}$. Alternately (and equivalently), it is any compact set obtained as the intersection of finitely many closed half-spaces in $\EE^{n}$. A {\em supporting} hyperplane $H$ of a convex polytope $P$ is any hyperplane $H$ such that $H\cap P\ne \emptyset$ and $P$ lies entirely in one of the two closed half-spaces determined by $H$. A {\em face} $F$ of $P$ is the intersection of a supporting hyperplane and $P$, or the empty set or $P$. The faces of $P$ form a poset under inclusion; this is a lattice called the {\em face lattice\/} of $P$, and is denoted {\mc L}(P). The maximal chains of faces of $P$ are called {\em flags}. Two flags are said to be {\em adjacent} if they differ by only a single element (face); if this face is an $i$-face, the two flags are {\em $i$-adjacent}. A convex polytope $P$ is {\em flag connected}, meaning that for every pair of flags $\Phi, \Psi$ of $P$  there exists a finite sequence of flags $\Phi=\Phi_{0},\Phi_{1},\ldots,\Phi_{k}=\Psi$ such that $\Phi_{i}$ is adjacent to $\Phi_{i+1}$ for $0\le i\le k-1$. A {\em section} $G/F$ of a convex polytope $P$ is the collection of all faces $H$ of $P$ such that $F\le H\le G$. A convex polytope $P$ is even {\em strongly flag connected}, meaning that every section of $P$ (including $P$ itself) is flag connected. The face lattice of a convex polytope $P$ satisfies the following conditions (and many others) \cite{Zie95,McMSch02}:
\begin{enumerate}
\item $P$ has a least face ($\emptyset$) and a greatest face ($P$ itself).
\item Every flag of $P$ has exactly $n+1$ elements in it.
\item $P$ is strongly flag-connected.
\item Given faces $F<G$ of $P$ such that ${\rm rank}(G)-{\rm rank}(F)=2$, there exist exactly two distinct faces $H_{1},H_{2}$ such that $F<H_{1},H_{2}<G$.
\end{enumerate}
This last property is known as the {\em diamond condition}. 

Any partially ordered set whose elements (which we will call {\em faces}) satisfy the four conditions above is an {\em abstract polytope of rank $n$}. Thus convex polytopes are abstract polytopes. Note that any section of an abstract polytope is itself an abstract polytope with rank given by $({\rm rank}(G)-{\rm rank}(F)-1)$, with $F$ playing the role of the empty set and $G$ playing the role of the maximal face. For convex polytopes we use the terms ``rank" and ``dimension" interchangeably. 

{A \emph{vertex} of a convex or abstract polytope is a face of rank 0. A \emph{vertex figure} of a convex $n$-polytope $P$ at a vertex $v$ is a convex $(n-1)$-polytope obtained by intersecting $P$ with a hyperplane~$H$ that separates $v$ from the other vertices of $P$. Irrespective of the choice of such a separating hyperplane~$H$, the combinatorial type of a vertex figure at $v$ is the same; in fact, the faces of a vertex-figure are in one-to-one correspondence with the faces of $P$ that contain~$v$. Thus we often speak of {\it the} vertex-figure of $P$ at $v$. The \emph{vertex figure} of an abstract $n$-polytope $\mathcal{P}$ at a vertex $v$ is the partially ordered set consisting of all faces of $\mathcal{P}$ that are incident with $v$; this is an abstract polytope in its own right, of rank $n-1$. This terminology is consistent with that for convex polytopes:\ the common face lattice of the different vertex-figures at a vertex of a convex polytope is just the abstract vertex-figure of the face lattice of the convex polytope at that vertex. }

{The standard \emph{barycentric subdivision} of a convex $n$-polytope $P$ is the geometric simplicial complex of dimension $n$, whose $n$-simplices are precisely the convex hulls of the centroids of the faces in flags of $P$ (see \cite{Bay88,GooORo04} or \cite[Sect. 2C]{McMSch02}). We will use the term ``barycentric subdivision" more broadly and allow the centroid of a face to be replaced by a relative interior point of that face. Thus a barycentric subdivision of $P$ is an $n$-dimensional geometric simplicial complex with one vertex in the relative interior of each non-empty face of $P$, and with one $n$-dimensional simplex per flag of $P$, such that the vertices of an $n$-simplex are precisely the relative interior points chosen in the faces in the corresponding flag. Each barycentric subdivision of $P$ is isomorphic (as an abstract simplicial complex) to the order complex of the face lattice of $P$; in particular, any two barycentric subdivisions are isomorphic.

There is a similar notion of barycentric subdivision for the boundary complex of a convex polytope. Recall that the \emph{boundary complex} of a convex $n$-polytope $P$, denoted $\rm{bd}(P)$, is the set of faces of $P$ of rank less than $n$, partially ordered by inclusion (see~\cite[p. 40]{Gru03}); this complex tessellates the boundary $\partial P$ of $P$ and is topologically a sphere. }

\section{The Construction}

The goal of this section is to establish Theorems~\ref{grpol} and~\ref{Pconvexly} stating (together) that any finite group $\Gamma$ is the automorphism group of a convex polytope.

{ It is convenient to settle the case of cyclic groups $\Gamma=C_k$ upfront. Figure~\ref{fig:C3} illustrates, for $k=3$, how for $k\geq 3$ the $1$-skeleton of a convex polyhedron in $\mathbb{E}^3$ with automorphism group and symmetry group $C_k$ can be constructed. The polyhedron itself may be achieved geometrically by an appropriate choice of points in $\mathbb{E}^3$. For $C_{2}$ we observe that any line segment exhibits this group as its group of either geometric or combinatorial automorphisms; and the trivial group, $C_1$, similarly occurs for a $0$-dimensional polytope, a point. Thus all cyclic groups occur as automorphism groups and symmetry groups of convex polytopes.}

\begin{figure}[htbp]
\begin{center}
\includegraphics[width=2in]{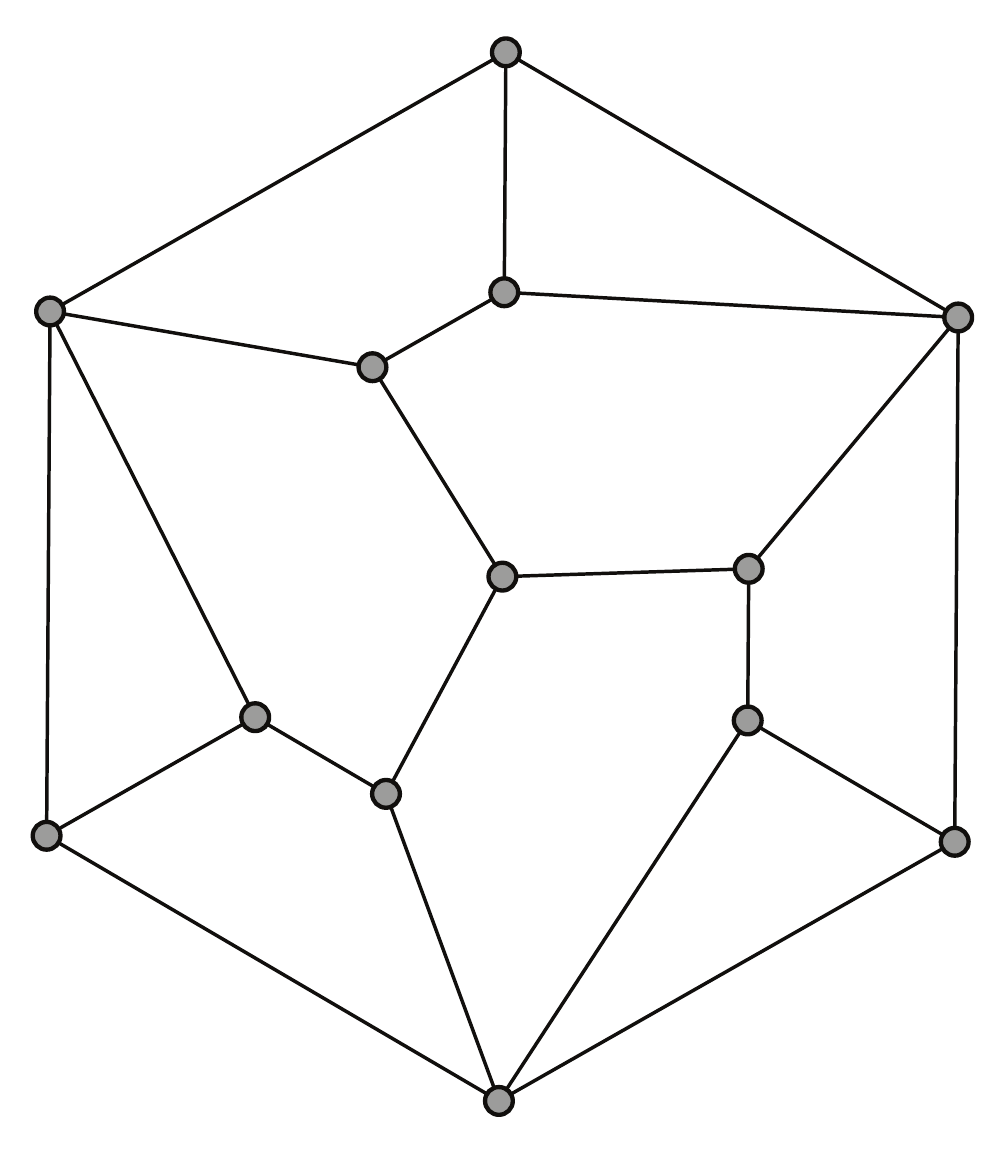}
\caption{The 1-skeleton of a polyhedron with automorphism and symmetry group $C_{3}$.}
\label{fig:C3}
\end{center}
\end{figure}

{Now let $\Gamma$ be a non-cyclic finite group, so $\Gamma$ has order at least $4$. (Our construction goes through for cyclic groups as well but some arguments become slightly more complicated.)} Then $\Gamma$ can be viewed as a subgroup of the symmetric group $S_{n+1}$ for some~$n$. 

Consider the $n$-simplex $T$ in $\mathbb{E}^{n+1}$ with vertices given by the canonical base points (vectors)  $e_1,\ldots,e_{n+1}$ in $\mathbb{E}^{n+1}$. Then $T$ lies in the hyperplane $E$ of $\mathbb{E}^{n+1}$ given by $\sum_{i=1}^{n+1}x_{i}=1$. The symmetry group $G(T)$ of $T$ (in $E$) is isomorphic to $S_{n+1}$ and hence can be identified with~$S_{n+1}$. Thus $\Gamma$ is a subgroup of $G(T)$.

Consider the barycentric subdivision $\mathcal{C}(T)$   of the boundary complex ${\rm bd}(T)$ of $T$ in $E$.  This simplicial $(n-1)$-complex is a refinement of ${\rm bd}(T)$, and its maximal simplices, called {\em chambers\/}, all have dimension $n-1$. The symmetry group $G(T)$ acts simply transitively on the chambers, and each chamber of $\mathcal{C}(T)$ is a fundamental region for the action of $G(T)$ on the boundary $\partial T$ of~$T$. The barycentric subdivision $\mathcal{C}(T)$ is a realization of the Coxeter complex for the Coxeter group $A_{n+1}$ (see \cite{McMSch02}, Sect. 3A,B). 

Let $C$ be a chamber of $\mathcal{C}(T)$, which we take to be the fundamental region for $G(T)$. Pick a point $v$, the {\em initial vertex\/}, in the relative interior of $C$ and define
\[ T' := {\rm conv} \{ \gamma(v) \mid \gamma \in  G(T) \} \]
and
\[ Q := {\rm conv} \{ \gamma(v) \mid \gamma \in  \Gamma \}. \]
Then $T'$ and $Q$ are convex polytopes in $E$, and $T'$ has full dimension $n$ in $E$. In the notation of Coxeter~\cite{Cox40,Cox85}, $T'$ is given by the diagram 
\begin{equation}
\label{2kgp}
\centering
\begin{picture}(180,30)
\put(100,0){
\multiput(15,0)(45,0){2}{\circle*{5}}
\multiput(-130,0)(45,0){2}{\circle*{5}}
\put(-130,0){\circle{9}}
\put(-85,0){\circle{9}}
\put(15,0){\circle{9}}
\put(60,0){\circle{9}}
\put(-5,0){\line(1,0){20}}
\put(-85,0){\line(1,0){20}}
\put(-130,0){\line(1,0){45}}
\put(15,0){\line(1,0){45}}
\put(4,9){\scriptsize $n-1$}
\put(57,9){\scriptsize $n$}
\put(-133,9){\scriptsize $0$}
\put(-87,9){\scriptsize $1$}
\put(-49,-0.5){$\ldots\ldots$}}
\end{picture}
\end{equation}
The convex polytope $Q$ may have dimension smaller than $n$. Let $d$ denote the dimension of $Q$.  Then $d\geq 2$, since $\Gamma$ has order at least $4$.

Then we have the following lemma.

\begin{lemma}\label{propQ}
$Q$ is a convex $d$-polytope in $E$ on which $\Gamma$ acts simply vertex-transitively as a group of symmetries. In particular, $\Gamma$ is a subgroup of the symmetry group $G(Q)$ of $Q$.
\end{lemma}

\begin{proof}
This follows directly from the fact that $G(T)$ acts simply transitively on the chambers of $\mathcal{C}$ and thus  also on the vertices of $T'$. The vertices of $Q$ are among those of $T'$.
\end{proof}

If the combinatorial automorphism group $\Gamma(Q)$ of $Q$, which contains $G(Q)$ as a subgroup, already coincides with $\Gamma$, then we are done, since then $Q$ is a convex (and hence abstract) polytope with $\Gamma(Q)=\Gamma$. However, this is generally not the case. In fact, there are simple examples where already the subgroup $G(Q)$ is strictly larger than $\Gamma$. 

{ We first dispose of the case when $Q$ lies in a 2-dimensional flat. Then $Q$ is a polygon and $\Gamma(Q)$ is a dihedral group. Every subgroup of a dihedral group is cyclic or dihedral. Hence, being non-cyclic, $\Gamma$ itself must be dihedral. But then $\Gamma=\Gamma(Q)$, by the vertex-transitivity, and we are done.}

The remainder of our construction applies to the cases when the dimension of Q is greater than or equal to 3.

Suppose from now on that $\Gamma$ is a proper subgroup of $\Gamma(Q)$, and $Q$ is a convex polytope of dimension 3 or higher. Clearly $\Gamma(Q)$ also acts transitively on the vertices of $Q$, and since $\Gamma$ is a proper subgroup, the stabilizer $\Gamma_v(Q)$ of the initial vertex $v$ of $Q$ in $\Gamma(Q)$ is nontrivial. Note that $\Gamma_v(Q)$ permutes the vertices adjacent to $v$ in $Q$, as well as the facets of $Q$ containing $v$. Our strategy is to alter (in fact, refine) the structure of $Q$ in such a way that all automorphisms in the vertex-stabilizer are destroyed. The result is a spherical (abstract) polytope whose automorphism group is given by $\Gamma$.

Consider the barycentric subdivision $\mathcal{C}(Q)$ of the boundary complex ${\rm bd}(Q)$ of $Q$ in $E$. Then $\mathcal{C}(Q)$ is a simplicial $(d-1)$-complex that refines ${\rm bd}(Q)$, and can be viewed as a realization of the order complex of ${\rm bd}(Q)$ (see~\cite{McMSch02}, Sect. 2C). The simplices in $\mathcal{C}(Q)$ correspond to chains (totally ordered subsets) in the poset ${\rm bd}(Q)$, with the chambers (maximal simplices) corresponding to the flags of ${\rm bd}(Q)$. In particular, $\mathcal{C}(Q)$ has the structure of a {\em labelled\/} simplicial complex, in which every simplex is labelled by the set of ranks of the faces in the chain of ${\rm bd}(Q)$ represented by the simplex. Thus the vertices of $\mathcal{C}(Q)$ can be labelled by integers from $0,\ldots,d-1$. The vertices of $\mathcal{C}(Q)$ with label $0$ are precisely the vertices of $Q$. 

The vertex-stars in $\mathcal{C}(Q)$ are topologically $(d-1)$-balls and hence the vertex-links are $(d-2)$-spheres. (Recall that the {\em vertex-star\/} of a vertex $x$ in a simplicial complex $\mathcal{S}$, denoted ${\rm star}_{\mathcal{S}}(x)$, consists of all the simplices that contain $x$, and all their faces. The {\em vertex-link\/} of $x$ in $\mathcal{S}$ consists of all simplices in ${\rm star}_{\mathcal{S}}(x)$ that do not contain $x$.) The vertex-star in $\mathcal{C}(Q)$ at a vertex of $Q$ is isomorphic to the barycentric subdivision of the vertex-figure of $Q$ at that vertex; the corresponding vertex-link in $\mathcal{C}(Q)$ is isomorphic to the barycentric subdivision of the boundary complex of this vertex-figure. 

Note that $\Gamma(Q)$ acts on $\mathcal{C}(Q)$ as a group of automorphisms of a simplicial complex; in fact, since $Q$ is a convex polytope, this action is even by PL-homeomorphisms (every automorphism can be realized by a PL-homeomorphism). Moreover, since $\Gamma(Q)$ acts freely on the flags of $Q$, it also acts freely on the chambers of $\mathcal{C}(Q)$. 

A key step in the construction consists of chamber replacement by complexes made up of Schlegel diagrams of convex polytopes (see \cite{Gru03}); in other words, these new complexes are inserted into the chambers of $\mathcal{C}(Q)$.  In a first step we are using Schlegel diagrams of $d$-crosspolytopes, as these allow us to perform certain tasks successively rather than simultaneously. 

The Schlegel diagram $\mathcal{D}$ of a $d$-crosspolytope has dimension $d-1$ and consists of an {\em outer\/} $(d-1)$-simplex $D$, tiled in a face-to-face manner by $(d-1)$-simplices, the {\em simplex tiles\/} of~$\mathcal{D}$, each corresponding to a facet of the crosspolytope distinct from the facet that defines~$D$. Among the simplex tiles of $\mathcal{D}$ is a {\em central\/} $(d-1)$-simplex $Z$, corresponding to the facet opposite to the facet defining $D$. Note that $D$ and $Z$ have no vertices in common.

We will exploit two basic properties of these Schlegel diagrams. First, the simplex tiles of $\mathcal{D}$ that are adjacent to $Z$ (i.e., intersecting $Z$ in a common facet) have precisely one vertex in common with $D$; conversely, every vertex $u$ of $D$ is a vertex of precisely one simplex tile, $F_u$ (say), that is adjacent to $Z$. Second, if $D'$ is any $(d-1)$-simplex in a Euclidean space, then any affine transformation that maps $D$ to $D'$ takes the given Schlegel diagram $\mathcal{D}$ to a new Schlegel-diagram $\mathcal{D}'$ affinely equivalent to~$\mathcal{D}$. Note here that we can arbitrarily preassign the images of the vertices of $D$ among those of $D'$ and still find a suitable affine transformation that transforms the vertices accordingly. This second property extends to any Schlegel diagram supported on an $(d-1)$-simplex.

The next step is to modify $\mathcal{D}$ in such a way that the vertices in the outer simplex $D$ acquire very high valencies compared with those in the interior, and that the valencies of the vertices of $D$ are very far apart from each other. To this end, let $u_0,\ldots,u_{d-1}$ be a fixed labeling of the vertices of $D$, and let $F_{u_0},\ldots,F_{u_{d-1}}$ be the corresponding simplex tiles with vertices $u_0,\ldots,u_{d-1}$, respectively, adjacent to $Z$. We now replace every simplex tile $F_{u_i}$ of $\mathcal{D}$ by the Schlegel diagram of an affine image of a suitable convex $d$-polytope $R_i$. All vertices of this polytope $R_i$, save one, have small valencies but the exceptional vertex has a large valency given by an integer $m_i$ yet to be determined. An example of a polytope of this kind is the pyramid over a {\it simple\/} convex $(d-1)$-polytope that has $m_i$ vertices and at least one facet which is a simplex; then the pyramid itself has a simplex facet with the apex as a vertex of valency $m_i$. Suppose $R_i$ is a pyramid of this kind. Then $R_i$ admits a Schlegel diagram $\mathcal{R}_i$ whose outer simplex corresponds to a simplex facet of $R_i$ that contains the apex. In this diagram, the outer vertex representing the apex has valency $m_i$ while all other vertices have (small) valency $d$. Now choose an affine transformation that maps the outer simplex of $\mathcal{R}_i$ to the simplex tile $F_{u_i}$ of $\mathcal{D}$, and  the vertex corresponding to the apex to the distinguished vertex $u_i$ of $F_{u_i}$. Then insert the corresponding affine image of the Schlegel diagram $\mathcal{R}_i$ into the simplex $F_{u_i}$, such that $F_{u_i}$ becomes the outer simplex. If this procedure is carried out for each $i=0,\ldots,d-1$, the result is a $(d-1)$-dimensional complex $\mathcal{R}$ supported on $D$, in which each vertex $u_i$ of $D$ has a large valency, namely $m_{i}+d-1$, while all vertices of $\mathcal{R}$ not in $D$ have small valencies. 

We require one additional type of modification, now targeting the $(d-1)$-simplex of $\mathcal{R}$ that was the central simplex of $\mathcal{D}$. If $L$ is a {\it simplicial} convex $d$-polytope, then we let $\mathcal{R}^L$ denote the $(d-1)$-dimensional complex supported on $D$, in which the central simplex has been replaced by a suitable affine copy of a Schlegel diagram of $L$.

Now recall that $Q$ is a convex $d$-polytope in $E$ on which the given finite group $\Gamma$ acts simply vertex-transitively as a group of symmetries. Consider the vertex-star ${\rm star}_{\mathcal{C}(Q)}(v)$ of the initial vertex $v$ of $Q$ in the barycentric subdivision $\mathcal{C}(Q)$ of the boundary complex ${\rm bd}(Q)$ of $Q$. The chambers in ${\rm star}_{\mathcal{C}(Q)}(v)$ are precisely the chambers of $\mathcal{C}(Q)$ that have $v$ as a vertex, so clearly these chambers are permuted among each other by the elements of the vertex stabilizer group $\Gamma_v(Q)$ of $v$ in the full automorphism group $\Gamma(Q)$ of $Q$. Recall also that $\mathcal{C}(Q)$ is a labelled simplicial complex in which the vertices of a chamber are labelled $0,\ldots,d-1$ such that no two vertices are labelled the same. 

Now, for each chamber $C$ in ${\rm star}_{\mathcal{C}(Q)}(v)$, choose a simplicial convex $d$-polytope $L_C$, in such a way that the polytopes for any two distinct chambers have different numbers of vertices (and hence are not combinatorially isomorphic). Thus each chamber $C$ in ${\rm star}_{\mathcal{C}(Q)}(v)$ gives us a $(d-1)$-dimensional complex $\mathcal{R}^{L_C}$ supported on $D$, obtained by inserting an affine copy of the Schlegel diagram of $L_C$ into the simplex of $\mathcal{R}$ that corresponds to the central simplex of $\mathcal{D}$. Note that no two complexes $\mathcal{R}^{L_C}$ are combinatorially isomorphic since their numbers of vertices are distinct.

In the final step of the construction we first replace each chamber $C$ in ${\rm star}_{\mathcal{C}(Q)}(v)$ by an affine copy of the corresponding complex $\mathcal{R}^{L_C}$ such that, for each $i=0,\ldots,d-1$, the vertex $u_i$ of $D$ is mapped onto the vertex of $C$ labelled $i$. We then exploit $\Gamma$ to carry this new structure to the vertex-stars ${\rm star}_{\mathcal{C}(Q)}(w)$ in $\mathcal{C}(Q)$ at other vertices $w$ of $Q$. More precisely, if $w$ is a vertex of $Q$, and $\gamma$ denotes the unique element of $\Gamma$ such that $\gamma(v)=w$, then $\gamma$ maps chambers of ${\rm star}_{\mathcal{C}(Q)}(v)$ to chambers of ${\rm star}_{\mathcal{C}(Q)}(w)$ while preserving the types $i$ of vertices of chambers for each $i=0,\ldots,d-1$.
Hence, if $C'$ is a chamber of ${\rm star}_{\mathcal{C}(Q)}(w)$ and $C'=\gamma(C)$ for some chamber $C$ of ${\rm star}_{\mathcal{C}(Q)}(v)$, then we replace $C'$ by an affine copy of the complex $\mathcal{R}^{L_C}$ that we used for $C$, such that, for each $i=0,\ldots,d-1$, the vertex $u_i$ of $D$ is mapped onto the vertex of $C'$ labelled $i$. In short, with respect to insertion of diagrams we treat $C$ and $C'$ equivalently. If we proceed in this way for every vertex $w$ of $Q$, the final result is a new $(d-1)$-dimensional complex $\mathcal{C}'$, which is a refinement of the barycentric subdivision $\mathcal{C}(Q)$ and has the full $(d-2)$-skeleton of $\mathcal{C}(Q)$ as a subcomplex, unrefined. In particular, $\mathcal{C}'$ tiles the boundary $\partial Q$ of $Q$ and hence is topologically a $(d-1)$-sphere. Moreover, by construction, $\Gamma$ acts on $\mathcal{C}'$ as a group of automorphisms. By adjoining suitable improper faces (of ranks $-1$ and $d$) to $\mathcal{C}'$ we obtain a spherical abstract $d$-polytope, denoted $\mathcal{P}$. 

We still need to describe the choice of the parameters $m_0,\ldots,m_{d-1}$. Let $Q$ be as before, and let $\mathcal{C}(Q)$ be its barycentric subdivision. { For a vertex $u$ of $\mathcal{C}(Q)$ let $s_u$ denote the number of chambers containing $u$; note that $s_{u}$ is just the number of flags of $Q$ containing  the face of $Q$ that corresponds to~$u$. If} $x$ is a vertex of any complex $\mathcal{S}$, we also write ${\rm val}_\mathcal{S}(x)$ for its valency in the edge graph of $\mathcal{S}$.

Now the valencies of the vertices in $\mathcal{P}$ (or $\mathcal{C}'$) are as follows. If $x$ is a vertex of $C(Q)$ of type $i$, then $x$ is a vertex of type $i$ in every chamber containing $x$ and therefore
\begin{equation}
\label{valxi}
{\rm val}_\mathcal{P}(x) = {\rm val}_{\mathcal{C}(Q)}(x) + s_xm_{i}.
\end{equation}
If $x$ is a vertex of the {\it central\/} simplex in the complex $\mathcal{R}^{L_C}$ inserted into a chamber $C$, then
\[ {\rm val}_\mathcal{P}(x) = 2(d-1) +({\rm val}_{L_C}(x) - (d-1)) = {\rm val}_{L_C}(x) + d-1 .\]
Further, ${\rm val}_\mathcal{P}(x)=d$, if $x$ is a vertex of the copy of $R_i$ that is not a vertex of a chamber $C$ of $\mathcal{C}(Q)$ or of the central simplex inside $C$; and ${\rm val}_\mathcal{P}(x)={\rm val}_{L_C}(x)$, if $x$ is a vertex of the copy of $L_C$ that is not a vertex of the central simplex in a chamber $C$. In particular, this shows that there exists a constant $m$ (depending on $d$ and the polytopes $L_C$) such that 
\begin{equation}
\label{emm}
{\rm val}_\mathcal{P}(x) \leq m 
\end{equation}
for all vertices $x$ of $\mathcal{P}$ which are not vertices of $\mathcal{C}(Q)$, independent of the choice of { the parameters} $m_{i}$. 

{ At this point we still have the choice of the parameters $m_0,\ldots,m_{d-1}$ at our disposal. Before we proceed with their definition we need one more piece of notation. Suppose for a moment that a specific parameter value $m_i$ has been chosen and then substituted in equation (\ref{valxi}) to provide certain vertex valencies in \mc P. In this situation we write $a_{i}$ and $b_{i}$ for the minimum or maximum valency in \mc P, respectively, taken over all vertices $x$ in $\mathcal{C}(Q)$ of type $i$, as given in (\ref{valxi}). Thus $a_{i}\leq {\rm val}_\mathcal{P}(x) \leq b_{i}$ for each vertex $x$ of $\mathcal{C}(Q)$ of type $i$. } 

Now the parameters $m_i$ are determined inductively for $i=d-1,d-2,\ldots,0$, beginning with $m_{d-1}:=m$ where $m$ is as in (\ref{emm}). {Using the notation just introduced, we then have $m< a_{d-1}\leq b_{d-1}$.} Next choose $m_{d-2}$ in such a way that $b_{d-1}<a_{d-2}$. More generally, {if $j\leq d-1$ and $m_{j}$ has already been chosen}, we pick $m_{j-1}$ in such a way that $b_{j}<a_{j-1}$. When $j=1$ this gives us $m_0$. Our choices guarantee that
\begin{equation}
\label{ab}
m< a_{d-1}\leq b_{d-1} < a_{d-2}\leq b_{d-2} <\ldots\ldots < a_{1}\leq b_{1}<a_{0}\leq b_{0}.
\end{equation}
In particular, if we set $M_{i}:=[a_{i},\,b_{i}]$ for each $i$, then the intervals $M_0,\ldots,M_{d-1}$ are mutually disjoint. 

\begin{lemma}
\label{aut}
If the parameters $m_0,\ldots,m_{d-1}$ are chosen in such a way that (\ref{ab}) holds, then $\Gamma(\mathcal{P})=\Gamma$.
\end{lemma}

\begin{proof}
By construction $\Gamma$ is a subgroup of $\Gamma(\mathcal{P})=\Gamma(\mathcal{C}')$. We need to establish that $\Gamma$ is the full automorphism group of $\mathcal{P}$. 

We first show that every automorphism of $\mathcal{P}$ is induced by an automorphism of $Q$. To this end
suppose $\gamma$ is an automorphism of $\mathcal{P}$. Then $\gamma$ maps vertices of $\mathcal{P}$ to vertices with the same valency in the edge graph of $\mathcal{P}$. The vertices of $\mathcal{P}$ corresponding to vertices of $\mathcal{C}(Q)$ have a higher valency than other vertices of $\mathcal{P}$ and hence must be permuted among each other by $\gamma$. In other words, $\gamma$ maps vertices of $\mathcal{C}(Q)$ to vertices of $\mathcal{C}(Q)$. By construction, for each $i=0,\ldots,d-1$, the valency of each vertex of $\mathcal{C}(Q)$ of type $i$ lies in $M_i$. Hence, since $M_i$ and $M_j$ are mutually disjoint for $i\neq j$, the automorphism $\gamma$ must necessarily map vertices of $\mathcal{C}(Q)$ of type $i$ to vertices of $\mathcal{C}(Q)$ of the same type, $i$, for each $i$. 

Now since the full $(d-2)$-skeleton of $\mathcal{C}(Q)$ is an (unrefined) subcomplex of the $(d-1)$-dimensional complex $\mathcal{C}'$, the automorphism $\gamma$ induces a vertex-type preserving automorphism of~$\mathcal{C}(Q)$; note here that the $(d-2)$-skeleton of~$\mathcal{C}(Q)$ already contains all the information about~$\mathcal{C}(Q)$, since only chambers need to be added to obtain the full complex (bear in mind that~$\mathcal{C}(Q)$ lives on a sphere). In particular, since the vertices of $\mathcal{C}(Q)$ corresponding to vertices of $Q$ are precisely the vertices of $\mathcal{C}(Q)$ of type $0$, the vertices of $Q$ then must be permuted by $\gamma$. 

Thus $\gamma$ induces a vertex-type preserving automorphism of $\mathcal{C}(Q)$ mapping vertices of $Q$ to vertices of $Q$. It follows that $\gamma$ induces an automorphism of $Q$ itself. In fact, every face of $Q$ is uniquely determined by the set of flags of $Q$ containing this face, so every vertex of $\mathcal{C}(Q)$ is uniquely determined by the chambers of $\mathcal{C}(Q)$ containing this vertex. Now if $F$ is an $i$-face of $Q$ and $w_F$ is the corresponding vertex of type $i$ in $\mathcal{C}(Q)$, then $\gamma(w_F)$ is also a vertex of type $i$ in $\mathcal{C}(Q)$ and hence corresponds to an $i$-face of $Q$. This $i$-face is simply $\gamma(F)$. Note here that $\gamma$ induces an isomorphism between the vertex-stars of $w_F$ and $\gamma(w_F)$ in $\mathcal{C}(Q)$; in particular, chambers of $\mathcal{C}(Q)$ containing $w$ are mapped in a one-to-one and type preserving manner to chambers containing $\gamma(w_F)$. While these arguments hold on a purely combinatorial level, it helps to observe that, since $\mathcal{C}(Q)$ is spherical, all combinatorial automorphisms can be realized by PL-homeomorphisms.

At this point we know that $\gamma$ arises from an automorphism $\gamma_Q$ (say) of $Q$. We need to show that $\gamma_Q$ determines $\gamma$ uniquely. 

Suppose $\gamma_Q$ is the identity map on $Q$. Then the automorphism induced by $\gamma$ on $\mathcal{C}(Q)$, $\gamma_{\mathcal{C}(Q)}$ (say), is also the identity map on $\mathcal{C}(Q)$, since the simplices in $\mathcal{C}(Q)$ just correspond to the chains of the boundary complex of $Q$, with vertices of $\mathcal{C}(Q)$ of type $i$ corresponding to faces of rank~$i$. Now with regards to the insertion of complexes like $\mathcal{R}^{L_C}$ into chambers of $\mathcal{C}(Q)$, observe that $\gamma$ maps a complex like $\mathcal{R}^{L_C}$ inserted into a chamber to a similar such complex inserted into the image chamber under $\gamma$. Note here that, since $\gamma$ is realizable by a homeomorphism, $\gamma$ maps ``interiors" of chambers to ``interiors" of chambers, so in particular~$\gamma$ will not swap ``interiors" and ``exteriors" of chambers. Now since $\gamma$ fixes every face of a chamber of $\mathcal{C}(Q)$, which in a complex like $\mathcal{R}^{L_C}$ becomes the outer simplex, then it must fix the entire complex inserted into the chamber. In fact, the outer simplex of a complex $\mathcal{R}^{L_C}$ can be joined to every tile in $\mathcal{R}^{L_C}$ by a sequence of successively adjacent tiles. Beginning with the outer simplex we then can move along this sequence to show that $\gamma$ is also the identity map on the face lattice of the target tile in the sequence; in fact, since consecutive tiles meet in a facet, it is clear that if $\gamma$ is the identity map on the face lattice of a tile in the sequence, then $\gamma$ is also the identity map on the face lattice of the next tile in the sequence. Hence $\gamma$ is the identity map on the entire complex $\mathcal{C}'$ and therefore also on $\mathcal{P}$. 

At this point we know that $\Gamma(\mathcal{P})$ can be viewed as a subgroup of $\Gamma(Q)$ containing the given group $\Gamma$. The final step consists of showing that $\Gamma(\mathcal{P})=\Gamma$.

Recall that $\Gamma$ acts simply transitive on the vertices of $Q$ (even when $\Gamma$ is viewed as a subgroup of $\Gamma(\mathcal{P})$). Hence $\Gamma(\mathcal{P})$ must also act transitively on the vertices of $\mathcal{P}$ which are also vertices of $Q$. We need to show that the stabilizer of a vertex of $Q$ in $\Gamma(\mathcal{P})$ is trivial. Let again $v$ denote the initial vertex of $Q$, and let $\gamma$ belong to the stabilizer of $v$ in $\Gamma(\mathcal{P})$. By our previous arguments, $\gamma$ induces an automorphism of $\mathcal{C}(Q)$ with $\gamma(v)=v$. Hence $\gamma$ induces an automorphism on the vertex-star ${\rm star}_{\mathcal{C}(Q)}(v)$ of $v$ in $\mathcal{C}(Q)$ and must necessarily  also permute the complexes $\mathcal{R}^{L_C}$ inserted into the chambers $C$ of this vertex-star. However, by construction, no two of these complexes $\mathcal{R}^{L_C}$ are isomorphic, so $\gamma$ must necessarily leave every chamber in $\mathcal{C}(Q)$ invariant, in a type preserving manner. Thus $\gamma$ must be the identity, and we are done.
\end{proof}

In summary, we have established the following theorem.

\begin{theorem}
\label{thm}
\label{grpol}
Every finite group is the automorphism group of a finite abstract polytope. In particular, if $\Gamma$ is isomorphic to a subgroup of the symmetric group $S_{n+1}$, then there exists a finite abstract polytope $\mathcal{P}$ of rank $d$, $d\leq n$, with automorphism group $\Gamma$, such that $\mathcal{P}$ is isomorphic to a face-to-face tessellation of the  $(d-1)$-sphere by topological copies of convex $(d-1)$-polytopes.
\end{theorem}

We now turn our attention to the realizability of $\po$ as the face lattice of a convex polytope. To do this we must establish the following result, which is of interest in its own right.

\begin{lemma}\label{realizeBSD} Let $R$ be a convex polytope with boundary complex $bd(R)$. Then there exists a convex polytope $R'$ such that $bd(R')$ is combinatorially equivalent to ${\mc C}(R)$, the barycentric subdivision of ${\rm bd}(R)$.
\end{lemma}

Once we have proved Lemma~\ref{realizeBSD} we may easily establish:

\begin{theorem}\label{Pconvexly}
The abstract polytope \mc P may be realized convexly.
\end{theorem}

\begin{proof} This follows easily from Lemma \ref{realizeBSD}  with $R=Q$. To see this, first observe that by Lemma \ref{realizeBSD} the barycentric subdivision $\mc C(Q)$ of ${\rm bd}(Q)$ may be realized convexly, and all subsequent modifications to the boundary required for the construction of \mc P may be achieved by gluing projective copies of convex polytopes to the facets of this convex realization of $\mc C(Q)$ that are sufficiently thin in the direction of the outward facing normal to the facet. {(See \cite{Zie95} for basic properties of projective transformations and their use in convex polytope theory.)}
\end{proof}

We now present the proof of Lemma \ref{realizeBSD}. {Given a convex polytope $P$, recall that a point $p$ is said to lie {\em below}, or {\em above}, a supporting hyperplane $H$ of $P$ if $p$ lies in the open halfspace bounded by $H$ that contains the interior of $P$, or does not meet $P$, respectively.}

\begin{proof}
We will prove this lemma by presenting an algorithm for constructing the desired polytope $R'$ from the given polytope $R$.

Suppose $R$ is a $k$-polytope with vertex set $V$ and face lattice \mc L$(R)$. Without any loss of generality let its centroid be the origin, $o$. For a non-empty face $F$ of $R$ let $b_F$ denote its centroid. For $j=0,\ldots k-1$ define
\[ {\mc B}_{j} := \{b_{F}\mid F\in  {\mc L}(R), {\rm rank}(F)=j \},\]
and set
\[ {\mc B} := \{b_{F}\mid F\in  {\mc L}(R), F\neq\emptyset,R\} \,=\, \bigcup_{j=0}^{k-1}{\mc B}_{j}. \]
 Then $\mc B$ is the vertex set of the barycentric subdivision ${\mc C}(R)$ of ${\rm bd}(R)$, and $\mc B\cup \{o\}$ is the vertex set of the complete barycentric subdivision of $R$.

We will proceed iteratively, working down from $j=k-1$ to $j=1$. Beginning with the original polytope $R$, we construct, in each step, a new convex $k$-polytope that has more vertices in ``convex position" than the previous polytope. The result is a sequence of convex $k$-polytopes 
\[R=:R_{k},R_{k-1},\ldots,R_{2},R_{1},\] 
in which the last entry, $R_1$, gives the desired polytope $R'$. In the course of the construction we introduce certain families of supporting hyperplanes $\mc H_{F}$ of these polytopes; these are indexed by faces $F$, which in the first step are of rank $k-1$ and in each subsequent step decrease in rank by $1$.

We first describe how to treat the vertices of ${\mc C}(R)$ lying in ${\mc B}_{k-1}$. Let $\mc H_{k-1}$ be the collection of supporting hyperplanes of the  facets (that is, $(k-1)$-faces) of $R$, and let $H_F$ denote the supporting hyperplane of a facet $F$ of $R$.  

 In the first step the families $\mc H_{F}$ to be defined are trivial (and just consist of $H_F$), unlike in subsequent steps. More explicitly, for each facet $F$ of $R$, let 
\[\mc H_{F}:=\{H\in \mc H_{k-1}\mid F\subseteq H\}\] 
and
\[{\mc H}_{F}^{-}:=\{H\in\mc H_{k-1}\mid \mbox{$H\cap F$ is a non-empty proper face of $F$}\}. \]
Clearly, ${\mc H}_{F}=\{H_F\}$. In order to construct the first new polytope, $R_{k-1}$, we may then choose points $b_{F}^{*}$ corresponding to each of the centroids $b_{F}$, with $F$ a facet of~$R$,  such that 
\begin{enumerate}
\item $b_{F}^{*}$ is above each $H\in{\mc H}_{F}$ (that is, above $H_{F}$),  
\item below each $H\in{\mc H}_{F}^{-}$, 
\item and so that the line segment connecting any two such points $b_{F}^{*}$ has a nonempty intersection with the interior  of $R=R_k$ (it may be helpful to think of the point $b_{F}^{*}$ as lying on the ray $\overrightarrow{o\,b_{F}}$).
\end{enumerate}
Then set
\[ \mc B_{k-1}^{*}:= \{b_{F}^{*}\mid F \mbox{ a facet of } R\} \]
and
\[ \hat{\mc B}_{k-1}:=(\mc B\setminus \mc B_{k-1})\cup \mc B_{k-1}^{*}, \]
 and define the polytope $R_{k-1}$ by
\[R_{k-1} := {\rm conv}(\hat{\mc B}_{k-1}) = {\rm conv}(V\cup \mc B^{*}_{k-1}).\]
 By construction, all points in $\hat{\mc B}_{k-1}$ are vertices of $R_{k-1}$, and each facet of $R_{k-1}$ contains exactly one point in $\mc B_{k-1}^{*}$, all thanks to the constraints imposed by (1), (2) and (3)  above. Also note that the $(k-2)$-skeleton of  the original polytope $R$ is a subcomplex of the $(k-2)$-skeleton of the new polytope $R_{k-1}$, so in particular the centroids in $\mc B\setminus\mc B_{k-1}$ are also centroids of faces of $R_{k-1}$.
 
For each successive $j = k-2,\ldots,1$ (performed in decreasing order), we repeat a similar construction step to find a new polytope $R_{j}$ from the already constructed polytope $R_{j+1}$. To this end we let $\mc H_{j}$ be the collection of supporting hyperplanes of { the facets of} $R_{j+1}$, and for each $j$-face $F$ of $R_{j+1}$ also belonging to $R$ we let 
\[\mc H_{F}:=\{H\in \mc H_{j}\mid F\subseteq H\}\] 
and 
\[\mc H_{F}^{-}:=\{H\in \mc H_{j}\mid \mbox{$H\cap F$ is a non-empty proper face of $F$}\}.\] 

From these families of hyperplanes we may similarly find a collection of points $b_{F}^{*}$, with $F$ a $j$-face of $R_{j+1}$ belonging to $R$, satisfying the following slight modifications to the constraints we had before on the  points $b^{*}_{F}$:
\begin{enumerate}[1$^{*}$.]
\item $b_{F}^{*}$ is above each $H\in{\mc H}_{F}$,  
\item below each $H\in{\mc H}_{F}^{-}$, 
\item and so that the line segment connecting any two such points $b_{F}^{*}$ has a nonempty intersection with the interior of $R_{j+1}$.
\end{enumerate}
From these choices for the points $b_{F}^{*}$ we obtain new sets of vertices $\mc B_{j}^{*}$ and $\hat{\mc B}_{j}$ and a convex polytope $R_{j}$, with the only vertices in $R_{j}$ not already in $R_{j+1}$ being the points in ${\mc B}_{j}^*$. Also note that at each step, all the faces of $R_{j+1}$ of rank lower than $j$ are preserved for $R_j$, and all of the newly introduced faces involve exactly one of the points in $\mc B_{j}^{*}$.

At the final stage of the process we arrive at a polytope $R_1$. Now let $R':=R_{1}$. Note that the only faces of $R$ present in the boundary complex of $R'$ are the vertices of $R$, for in the construction of $R_{j}$ we removed the faces of the boundary complex of $R_{j+1}$ of rank $j$ (and thus of $R$) by placing a point $ b^{*}_{F}$ above the hyperplanes defining that face. Moreover, the facets of $R'$ are $(k-1)$-simplices whose vertices are given by the points $b_F^*$ associated with the faces $F$ in a flag of $R$; in fact, a new 
point $b_F^*$, arising from a $j$-face $F$ of $R$ (and hence of $R_{j+1}$), is introduced at each step of the construction process, and if an $i$-face $G$ of $R$ is such that $G\leq F$ then $b_{G}^*$ and $b_{F}^*$ lie in a common facet of $R'$. Finally, note that every facet of $R'$ contains exactly one vertex of $R$ since the final step of the construction introduces vertices that eliminate edges connecting vertices of $R$; in particular, the vertex at the other end of such a new edge of $R'$ is below every facet defining hyperplane of $R$ that does not contain the original edge.
\end{proof}

\section{Concluding Remarks and Open Problems}

In the context of the current work, a number of interesting problems suggest themselves. While the construction used to prove Theorem~\ref{thm} would generally be expected to result in spherical (and convex) polytopes of arbitrarily large ranks, high ranks are not generally required if non-spherical abstract polytopes are permitted as the underlying combinatorial structures on which a given finite group $\Gamma$ acts as the full automorphism group. In fact, it was established in~\cite{sisk} that every finite group $\Gamma$ occurs as the automorphism group of an abstract polytope of rank $3$ (see also \cite{coma}). The construction in \cite{sisk} is described in terms of maps on surfaces but also proves the result for $3$-polytopes; the key idea is to start from a surface embedding of a Cayley graph of $\Gamma$ and then replace small neighborhoods of the edges by suitable planar graphs that force the resulting map to have automorphism group $\Gamma$. Thus the minimum possible rank, $3$, is achieved if arbitrary abstract polytopes are permitted to represent $\Gamma$. Note that only dihedral groups occur as automorphism groups of abstract polytopes of rank $2$.

The result in \cite{sisk} was inspired by a classical result obtained in~\cite{fru} stating that every finite group $\Gamma$ is the automorphism group of a finite graph. 

Our insistence on sphericity severely limits the choice of groups $\Gamma$ as automorphism groups of abstract 3-polytopes (or maps on surfaces). For a discussion of the groups that can act on the 2-sphere see~\cite[Section 6.3.2]{grtu}. 

Call an abstract $4$-polytope {\em locally toroidal\/} if each facet or vertex-figure that is not a spherical 3-polytope, is 
a toroidal 3-polytope. The following problem seems to be open.
\begin{open}
Is every finite group the automorphism group of a finite locally toroidal abstract $4$-polytope?
\end{open}

One may also wish to modify the construction of Theorem~\ref{thm} to require that the polytopes produced have certain geometric or combinatorial features.
\begin{open}Can one modify the construction of Theorem~\ref{thm} so as to have $\mc P$ simplicial or $P$ cubical (all facets are isomorphic to cubes)?
\end{open}

It is also natural to ask if further restrictions on the classes of finite groups might have an impact upon structural features of the polytopes whose automorphism groups fall into the specified categories. For example, what can one say about finite abelian groups? 

One may also wish to consider the case of infinite groups. By their nature, abstract polytopes are necessarily countable, and so then their automorphism groups must also be countable. The countability of abstract polytopes is a consequence of the flag-connectedness. In fact, every flag can be reached from a given flag $\Phi$ by a finite sequence of flags in which successive flags are adjacent, and the number of flags that can be reached from $\Phi$ by a sequence of a given length $m$ is finite; as $m$ ranges over all non-negative integers, the set of flags is countable, and so is the set of faces. 

Tilings of $\EE^{n}$ are non-finite abstract polytopes and their symmetry groups have been the subject of intensive investigation (e.g., \cite{ConBurGoo08,GruShe89}). Less is known about their combinatorial automorphism groups. Call a tiling of $\EE^{n}$ by topological $n$-polytopes (homeomorphic copies of convex $n$-polytopes) {\em locally finite} if every point in space has a neighborhood meeting only finitely many tiles; call it {\em face-to-face\/} if any two tiles meet in a common face, possibly the empty set.

\begin{open} Which countable infinite groups are automorphism groups of locally finite face-to-face tilings of a finite-dimensional real space by topological polytopes (homeomorphic copies of convex polytopes)? How about finitely generated infinite groups?
\end{open}

It would also be interesting to explore to what extent epimorphisms between groups correspond to coverings between certain types of polytopes. The results from \cite{MPW14,MonSch14} may be helpful in studying this question.

\begin{acknowledgements}
The authors would like to acknowledge Branko Gr\"unbaum for providing useful comments during the preparation of this paper; Tom Tucker for some helpful discussion about group actions on maps; and finally the referees for making valuable suggestions that have improved the paper. The first author was supported by NSA-grant H98230-14-1-0124.
\end{acknowledgements}

\bibliographystyle{spmpsci}      


\end{document}